\documentclass[11pt]{amsart}

\usepackage{amsmath}
\usepackage{amssymb}

\newtheorem{thm}{Theorem}[section]
\newtheorem{cor}[thm]{Corollary}
\newtheorem{lem}[thm]{Lemma}
\newtheorem{propo}[thm]{Proposition}
\newtheorem*{A}{Theorem A}
\newtheorem*{B}{Theorem B}
\newtheorem*{Gonchar}{Gonchar's Lemma}

\theoremstyle{definition}

\theoremstyle{remark}
\newtheorem{rem}{Remark}

\newcommand{\capp}{\mbox{\rm cap}\,}
\newcommand{\dst}{\displaystyle}
\begin{document}

%%%%%%%%%%%%%%%%%%%%%%%%%%%%%%%%%%%%%%%%%%%%%%%%%%%%%%%%%%%%%%%%%%%%%%%%%%%%%%%%%%%%%%%%%%%%%%%%%%%%%%%%%%%%%%%%%%%%%%%%%%%%%%%%%%%%%%%%%%%%%%%%%%%%%%%%%%%%%%%%%%%%%%%%%%%%%%%%%%%%%%%%%%%%%%%%%%%%%%%%%%%%%%%%%%%%%%%%%%       TITULO      %%%%%%%%%%%%%%%%%%%%%%%%%%%%%%%%%%%%%%%%%%%%%%%%%%%
%%%%%%%%%%%%%%%%%%%%%%%%%%%%%%%%%%%%%%%%%%%%%%%%%%%%%%%%%%%%%%%%%%%%%%%%%%%%%%%%%%%%%%%%%%%%%%%%%%%%%%%%%%%%%%%%%%%%%%%%%%%%%%%%%%%%%%%%%%%%%%%%%%%%%%%%%%%%%%%%%%%%%%%%%%%%%%%%%%%%%%%%%%%%%%%%%%%%%%%%%%%%%%%%%%%%%%%%%%

\title[Meromorphic continuation of functions]
 {Meromorphic
 continuation of functions and arbitrary distribution of interpolation points}

\author[M. Bello Hern\'andez]{M. Bello Hern\'andez}
\address{Departamento de Matem\'aticas y Computaci\'on, Universidad de La Rioja,
c/ Luis de Ulloa s/n, 26004 Logro\~{n}o, La Rioja, Spain.}
\email{mbello@unirioja.es}

\author[B. de la Calle Ysern]{B. de la Calle Ysern}
\address{Departamento de Matem\'atica Aplicada, E.T.S. Ingenieros
Industriales, Universidad Polit\'ecnica de Madrid, Jos\'e G.
Abascal 2, 28006 Madrid, Spain.}
\email{bcalle@etsii.upm.es}
\thanks{This work was supported by
Direcci\'on General de Investigaci\'on, Ministerio de Educaci\'on y
Ciencia, under grant MTM2009-14668-C02-02. The second author also
received support from Universidad Polit\'ecnica de Madrid through
Research Group ``Constructive Approximation Theory and
Applications".}

%\date{\today}

% 41A20- approximation by rational functions
% 41A21- Padé approximation
% 41A25- Rate of convergence, degree of approximation
% 41A27- Inverse theorems
\begin{abstract}
We characterize the region of meromorphic continuation of an
analytic function $f$ in terms of the geometric rate of convergence
on a compact set of sequences of multi-point rational interpolants
of $f$. The rational approximants have a bounded number of poles and
the distribution of interpolation points is arbitrary.
\end{abstract}

\maketitle
%%%%%%%%%%%%%%%%%%%%%%%%%%%%%%%%%%%%%%%%%%%%%%%%%%%%%%%%%%%%%%%%%%%%%%%%%%%%%%%%%%%%%%%
%%%%%%%%%%%%%%%%%%%%%%%%%%%%%%%%%%%%%%%%%%%%%%%%%%%%%%%%%%%%%%%%%%%%%%%%%%%%%%%%%%%%%%%
%%%%%%%%%%%%%%%%%%%%%%%%%%%%%%%%%%%%%%%%%%%%%%%%%%   INTRODUCCION  %%%%%%%%%%%%%%%%%%%%%%
%%%%%%%%%%%%%%%%%%%%%%%%%%%%%%%%%%%%%%%%%%%%%%%%%%%%%%%%%%%%%%%%%%%%%%%%%%%%%%%%%%%%%%%%
%%%%%%%%%%%%%%%%%%%%%%%%%%%%%%%%%%%%%%%%%%%%%%%%%%%%%%%%%%%%%%%%%%%%%%%%%%%%%%%%%%%%%%%%

\section{Introduction}
\label{intro}

\subsection{Rational functions of best approximation}
It is well known that a sequence of polynomial or rational
approximants converging on a certain region at geometric rate often
converges in a larger region, giving thus additional information on
the analytic or meromorphic continuation of the limit function.
Walsh called that phenomenon overconvergence of the approximants
(see, for instance, \cite{walsh1,walsh2}). Notice that the same word
is used in a somewhat different setting to describe a property of
power series (see \cite{remmert}, Chapter 11). One of the most
beautiful theorems of this type was proved by Gonchar culminating a
series of results given by Bernstein, Walsh, and Saff, among others.

Let $K$ be a compact set of the complex domain $\mathbb{C}$ and $G$
the unbounded component of $\overline{\mathbb{C}}\setminus K$. We
say that $K$ is regular if the domain $G$ is regular with respect to
the Dirichlet problem. By $\capp(K)$ we mean the logarithmic
capacity of $K$. If $\capp(K)>0$, there exists the generalized Green
function of $G$ with pole at $z=\infty$ which we denote by
$g_G(z,\infty)$ (see \cite{totik}, Section II.4). The fact that $K$
is regular implies $\capp(K)>0$ .

Let $\mathcal{P}_n$ be the set of all complex polynomials of degree
at most n. Set
$$
\mathcal{R}_{n,m}=\left\{r\,:\,r=p/q,\,\, \,
p\in\mathcal{P}_{n},\,\, \, q\in\mathcal{P}_m,\,\, \, q\not\equiv
0\, \right\}
$$
and
$$
 d_{n,m}(K)=\min\, \{\|f-r\|_K\,:\, r\in\mathcal{R}_{n,m}\},
$$
where $\|\cdot\|_{K}$ stands for the sup norm on $K$.

\begin{A}
Let $f$ be a function defined on a regular compact set
$K\subset\mathbb{C}$. Then,
\begin{equation}\label{rate}
\limsup_{n\to\infty}\sqrt[n]{d_{n,m}(K)}\le 1/\theta<1
\end{equation}
if and only if the function $f$ admits meromorphic continuation with
at most $m$ poles on the set $E(\theta)=\{z\in\mathbb{C}\,:\,\exp\{
g_G(z,\infty)\}<\theta\}.$
\end{A}
In particular, the best rational approximation characterizes the
largest set $E(\theta)$ on which the function $f$ admits meromorphic
continuation with at most $m$ poles. That is, if $\theta_{f,m}$
denotes the supremum of the numbers $\theta$ such that the function
$f$ admits meromorphic continuation with at most $m$ poles on
$E(\theta)$, then
$$
\theta_{f,m}=1/\dst\limsup_{n\to\infty}\sqrt[n]{d_{n,m}(K)}.
$$

Gonchar proved Theorem A in \cite{gonsaff}. The essential
contribution by Gonchar was
 to prove that \eqref{rate} implies the assertion about the meromorphic
continuation of $f$. The reciprocal statement basically follows from
work by Walsh \cite{walsh2}. The first result in the direction of
Theorem A was given by Bernstein \cite{bernstein} in 1912 for $m=0$
and $K=[-1,1]$. The general case for $m=0$ was proved by Walsh
\cite{walsh}. Under the assumption that $K$ is bounded by a Jordan
curve, Saff proved Theorem A in \cite{saff}.

Let $f$ be an analytic function on a neighborhood of a compact set
$\Sigma$ wherein interpolation is carried out along an arbitrary
table and let $K$ be a regular compact set on which $f$ is defined.
Our main goal in this work is to show that we can characterize the
largest region (of a given type) of meromorphic continuation of $f$
by means of the geometric rate of convergence on $K$ of its
multi-point rational interpolants. Thus, we considerably enlarge
both the class of situations and the regions in which we can deduce
meromorphic continuation of $f$ and show that the characterization
of the largest region of meromorphic continuation does not depend on
whether or not the interpolation table is extremal.

\subsection{Row sequences of Pad\'e approximants}
To fix ideas, let us consider the simplest case of rational
interpolants corresponding to classical Pad\'e approximation. Let
$f$ be a function analytic on a neighborhood of $z=0$ and let $n$
and $m$ be nonnegative integers. The Pad\'e approximant of type
$(n,m)$ for $f$ is defined as the unique rational function
$\pi_{n,m}=p_{n,m}/q_{n,m}$ verifying
\begin{itemize}
\item[-] $\deg p_{n,m}\le n,\,  \deg q_{n,m}\le m$, and $q_{n,m}\not\equiv 0.$

\item[-] $q_{n,m}(z) f(z)-p_{n,m}(z)= A\,z^{n+m+1}+\cdots,$
\end{itemize}
where rational functions are identified if they coincide after
cancellation of common factors from numerator and denominator.  The
problem of finding $\pi_{n,m}$ is reduced to that of solving a
system of linear equations whose coefficients can be expressed in
terms of the Taylor coefficients of $f$.

The table $\{\pi_{n,m}\},\,\, n, m\in\mathbb{Z}_+$, is called the
Pad\'e table of $f$. For each $m\in\mathbb{Z}_+$ we say that $D_m$
is the disk of $m$-meromorphy of $f$ if $D_m$ is the largest open
disk with center at zero into which $f$ can be continued as a
meromorphic function that has at most $m$ poles, counting
multiplicities. The radius of $D_m$ is denoted by $R_m$. The
following result is implicitly contained in \cite{gonrows}.

\begin{B}
Suppose that there exists $z\not=0$ such that
\begin{equation}\label{monte}
\limsup_{n\to\infty}|f(z)-\pi_{n,m}(z)|^{1/n}\le\frac{|z|}{R}<1.
\end{equation}
Then, $R_m\ge R$, that is, $f$ admits meromorphic continuation with
at most $m$ poles on the set $\{z\in\mathbb{C}\,:\,|z|<R\}$.
\end{B}

From \eqref{monte} it follows that the disk of m-meromorphy $D_m$ is
characterized by the rate of convergence of the Pad\'e approximants
$\pi_{n,m}$ to $f$ on a fixed point $z\not=0$. The proof relies
heavily on the fact that the table of interpolation points is
newtonian (cf. Proposition \ref{teorema1b} below). In the present
paper we focus on the analogous problem for multi-point Pad\'e
approximants where the interpolation points tend to an arbitrary
distribution and we prove that the phenomenon of overconvergence is
not limited to approximants with maximal rate of convergence. This
will be done in Section \ref{exact}, see Theorem \ref{teorema1}.
Section \ref{definitions} contains some definitions and auxiliary
material whereas in Section \ref{convergence} we characterize
certain generalized sets of $m$-meromorphy of $f$ in terms of
convergence in $\sigma$-content (for the definition, see Section
\ref{definitions}) of the multi-point Pad\'e approximants. This
implies that the region where the function $f$ is proved to admit
meromorphic continuation is the largest among those of a given type.
In the case that the table of interpolation points is newtonian, we
provide more detailed descriptions.

%%%%%%%%%%%%%%%%%%%%%%%%%%%%%%%%%%%%%%%%%%%%%%%%%%%%%%%%%%%%%%%%%%%%%%%%%%%%%%%%%%%%%%%%%%%%%%%
%%%%%%%%%%%%%%%%%%%%%%%%%%%%%%%%%%%%%%%%%%%%%%%%%%%%%%%%%%%%%%%%%%%%%%%%%%%%%%%%%%%%%%%%%%%%%
%%%%%%%%%%%%%%%%%%%%%%%%%%%%  DEFINICIONES  %%%%%%%%%%%%%%%%%%%%%%%%%%%%%%%%%%%%%%%%%%%%%%%%%%%%%%
%%%%%%%%%%%%%%%%%%%%%%%%%%%%%%%%%%%%%%%%%%%%%%%%%%%%%%%%%%%%%%%%%%%%%%%%%%%%%%%%%%%%%%%%%%%%%%%%%%%%%%%%%%%%%%%%%%%%%%%%%%%%%%%%%%%%%%%%%%%%%%%%%%%%%%%%%%%%%%%%%%%%%%%%%%%%%%%%%%%%%%%%%%%%%%%%

\section{Definitions and auxiliary results}\label{definitions}
\subsection{Potential theory}
Let $\Sigma$ be a compact set of the complex domain $\mathbb{C}$
with connected complement. Denote the domain
$\overline{\mathbb{C}}\setminus \Sigma$ by $\Omega$.

Let $\mu$ be a positive unit Borel measure supported on $\Sigma$.
The logarithmic potential of $\mu$ is denoted by $P(\mu;z)$ and is
equal to $$\int_\Sigma -\log|z-\zeta|\,d\mu(\zeta).$$

Set
$$
r_0=\inf_{z\in \Sigma}\exp \{-P(\mu;z)\}\ge 0
$$
and
$$
E_\mu(r)=\{z\in\mathbb{C} \, : \,\exp\{ -P(\mu;z)\}< r\},\quad
r>r_0.
$$
Since the function $\exp \{-P(\mu;\cdot)\}$ is upper semi-continuous
on $\mathbb{C}$, $E_\mu(r),\, r>r_0$, is a non-empty bounded open
set that, in general, does not contain the whole of $\Sigma$. By the
same token $\exp \{-P(\mu;\cdot)\}$ attains its maximum on compact
sets. So, for each compact set $K$ let us denote
$$
\rho_\mu(K)=\|\exp\{ -P(\mu;\cdot)\}\|_K.
$$
\begin{lem}\label{lema}
Given any open neighborhood $U$ of $\Sigma$ there exists $r_1>r_0$
such that $E_\mu(r_1) \subset U$.
\end{lem}
\begin{proof}
First, notice that $ \exp\{-P(\mu;z)\}\ge r_0 $ for all
$z\in\mathbb{C}$, due to the maximum principle for potentials (see
Corollary II.3.3 in \cite{totik}).

We may assume that $F=\mathbb{C}\setminus U\not=\emptyset$. As
$F\cap\Sigma=\emptyset$ and
$$
-P(\mu;z)=\log|z|+o(1),\quad z\to\infty,
$$
the function $\exp\{ -P(\mu;\cdot)\}$ attains its
minimum on $F$. Put
$$
r_1=\min_{z\in F}\exp\{ -P(\mu;z)\}=\exp\{ -P(\mu;z_1)\}\ge
r_0,\quad z_1\in F.
$$
It is clear that $E_\mu(r_1)\subset U$.  As $z_1\not\in\Sigma$, the
function $P(\mu;\cdot)$ is harmonic on a neighborhood of $z_1$ on
which cannot be constant because the complement of $\Sigma$ has only
one connected component. Therefore, by the minimum principle, there
exists a point $z_0$ verifying $\exp\{-P(\mu;z_0)\}<r_1$, which
implies that $r_1>r_0$.
\end{proof}

If $\capp(\Sigma)>0$ we have
\begin{equation}\label{green2}
g_{\Omega}(z,\infty)=-\log\capp(\Sigma)-P(\mu_\Sigma;z),\quad
z\in\mathbb{C},
\end{equation}
where $\mu_\Sigma$ is the equilibrium measure of the set $\Sigma$.

 Let $\mu_{n}$
and $\mu$ be finite positive Borel measures on
$\overline{\mathbb{C}}$. By $\mu_n \stackrel{
\ast}{\longrightarrow}\mu$, $n\rightarrow\infty$, we denote the
weak$^*$ convergence of $\mu_{n}$ to $\mu$ as $n$ tends to infinity.
This means that for every continuous function $f$ on
$\overline{\mathbb{C}}$ it holds
\[
\lim_{n \rightarrow \infty} \, \int f(x) \,  d\mu_{n}(x)\, =\, \int
f(x)\,  d\mu(x).
\]

If all the measures $\mu_n,\, n\in\mathbb{N}$, are supported on a
fixed compact set $F$ and $\mu_n \stackrel{
\ast}{\longrightarrow}\mu$, $n\rightarrow\infty$, the principle of
descent for potentials (cf. \cite{totik}, Theorem $1.6.8$) says that
\begin{equation}\label{descent}
\liminf_{n\to\infty} P(\mu_n;z_n)\ge  P(\mu;z_0),
\end{equation}
where the sequence $\{z_n\}_{n\in\mathbb{N}}$ tends to
$z_0\in\mathbb{C}$  as $n$ goes to infinity and
\begin{equation}\label{descent2}
\lim_{n\to\infty} P(\mu_n;z)=P(\mu;z),
\end{equation} uniformly on compact subsets of
$\mathbb{C}\setminus F$.

The fine topology on $\mathbb{C}$ is the coarsest topology on
$\mathbb{C}$ for which all superharmonic functions are continuous.
This is finer than ordinary planar topology. If $D$ is a connected
open set then the boundary of $D$ in the fine and Euclidean topology
coincide (see Theorem 5.7.9 in \cite{helms}).

\subsection{Convergence in $\sigma$-content}
Let $A$ be a subset of the complex plane $\mathbb{C}$. By
$\mathcal{U}(A)$ we denote the class of all coverings of $A$ by at
most a numerable set of disks. Set
$$
\sigma(A)=\inf\left\{\sum_{i\in I} |U_i|\,:\,\{U_i\}_{i\in
I}\in\mathcal{U}(A)\right\},
$$
where $|U_i|$ stands for the radius of the disk $U_i$. The quantity
$\sigma(A)$ is called the $1$-dimensional Hausdorff content of the
set $A$. This set function is not a measure but fulfills some good
properties like countable semiadditivity which, for instance, is not
satisfied by the logarithmic capacity.

Let $\{\varphi_n\}_{n\in\mathbb{N}}$ be a sequence of functions
defined on a domain $D$ and $\varphi$ another  function also defined
on $D$. We say that the sequence $\{\varphi_n\}_{n\in\mathbb{N}}$
converges in $\sigma$-content to the function $\varphi$ on compact
subsets of $D$ if for each compact subset $K$ of $D$ and for each
$\varepsilon >0$, we have
$$
\lim_{n\to\infty} \sigma\left(\{z\in K :
|\varphi_n(z)-\varphi(z)|>\varepsilon\}\right)=0.
$$
Such a convergence will be denoted by $\sigma$-$\lim_{n\to\infty}
\varphi_n = \varphi$ in $D$.

The next lemma was proved by Gonchar in \cite{gonfix}.
\begin{Gonchar} Suppose that the sequence $\{\varphi_n\}$ of functions defined
on a domain $D\subset \mathbb{C}$ converges in $\sigma$-content to a
function $\varphi$ on compact subsets of $D$. Then the following
assertions hold true:
\begin{itemize}
\item[i)] If each of the functions $\varphi_n$ is meromorphic in $D$
and has no more than $k<+\infty$ poles in this domain, then the
limit function $\varphi$ is also meromorphic (more precisely, it is
equal to a meromorphic function in $D$ except on a set of
$\sigma$-content zero) and has no more than $k$ poles in $D$.

\item[ii)] If each function $\varphi_n$ is meromorphic and has no more than $k<+\infty$ poles
 in $D$ and the function $\varphi$ is meromorphic and has exactly $k$ poles in
 $D$, then all $\varphi_n,\,n\ge N$, also have $k$ poles in $D$; the poles of $\varphi_n$
 tend to the poles $z_1,,\dots,z_k$ of $\varphi$ (taking account of their orders) and
 the sequence $\{\varphi_n\}$ tends to $\varphi$ uniformly on compact subsets of
 the domain $D^{\prime}=D\setminus \{z_1,,\dots,z_k\}$.
\end{itemize}
\end{Gonchar}

\subsection{Multipoint Pad\'e approximants}

Let $f$ be an holomorphic function on a neighborhood $V$ of the
compact set $\Sigma\subset\mathbb{C}$. Let us fix a
family of monic polynomials
\begin{equation}\label{interpolacion}
w_n(z)=\,\, \prod_{i=1}^{n}(z-\alpha_{n,i}),\qquad n\in\mathbb{N},
\end{equation}
whose zeros are contained in $\Sigma$. It is easy to verify that for
each pair of nonnegative integers $n$ and $m,\, n\ge m$, there
exists polynomials $P$ and $Q$ satisfying
\begin{itemize}
\item[-] $\deg P\le n-m,\,  \deg Q\le m$, and $Q\not\equiv 0.$

\item[-]  $(Q\,f-P)/w_{n+1}\in{\mathcal{H}}(V)$,  where
${\mathcal{H}}(V)$ denotes the space of analytic functions in $V$.
\end{itemize}
Any pair of such polynomials $P$ and $Q$ defines a unique rational
function $\Pi_{n,m} = P /Q $ which is called the multi-point Pad\'e
approximant of type $(n,m)$ of $f$.

By requiring $\deg P\le n-m$ we have slightly modified the usual
definition of a multi-point Pad\'e approximant of type $(n,m)$  in
an equivalent form which is more suitable for the purposes of the
paper. Let $\Pi_{n,m} = P_{n,m}/Q_{n,m}$ where $Q_{n,m}$ and
$P_{n,m}$ are polynomials obtained eliminating all common zeros and,
unless otherwise stated, normalizing $Q_{n,m}$ so that
\begin{equation} \label{normado} Q_{n,m}(z) = \prod_{|\zeta_{n,k}|
\leq 1} \left(z - \zeta_{n,k}\right) \prod_{|\zeta_{n,k}| > 1}
\left(1- \frac{z}{\zeta_{n,k}}\right).
\end{equation}

Let $R_{\mu,m},\,m\in\mathbb{Z}_+$, be the supremum of the numbers
$r> r_0$ such that $f$ admits meromorphic continuation with at most
$m$ poles on $E_\mu(r)$. Lemma \ref{lema} proves that $R_{\mu,m}>r_0$.
We define the set of $m$-meromorphy of $f$ relative to $\mu $ as the
set $E_\mu(R_{\mu,m})$ and we denote it by $D_{\mu,m}$. It is easy
to see that $f$ admits meromorphic continuation with at most $m$
poles on $D_{\mu,m}$.

For a given polynomial $p$, we denote by $\Theta_p$ the normalized
zero counting measure of $p$. That is,
$$
\Theta_p=\frac{1}{\deg p}\sum_{\xi:\, p(\xi)=0}\delta_{\xi}.
$$
The sum is taken over all the zeros of $p$ and $\delta_{\xi}$
denotes the Dirac measure concentrated at $\xi$. It is said that the
sequence of interpolation points given by the polynomials $
\{w_n\}_{n\in\mathbb{N}}$ has the measure $\mu$ as its asymptotic
zero distribution if
$$
\Theta_{w_n} \stackrel{ \ast}{\longrightarrow}\mu, \quad
n\rightarrow\infty.
$$

The following theorem is an analog of Montessus de Ballore's theorem
(see \cite{montessus}) for multi-point Pad\'e approximants with
arbitrary distribution of interpolation points. It was proved in
\cite{warner} (cf. \cite{wallin}) and constitutes a straightforward
generalization of a previous result by Saff \cite{saff2}.

\begin{thm}\label{thwarner} Let the measure $\mu$ be
the asymptotic zero distribution of the sequence of interpolation
points given by $ \{w_n\}_{n\in\mathbb{N}}$. Suppose that $f$ has
exactly $m$ poles on $D_{\mu,m}$. Then
\begin{itemize}
\item[i)] For all $n\ge n_0$, $\Pi_{n,m}$ has exactly $m$ poles
which converge according to their multiplicity to the poles of the
function $f$ as $n$ tends to infinity.
\item[ii)] For each compact subset $K\subset D_{\mu,m}$
not containing any
pole of the function $f$, it holds
\begin{equation}\label{montessus}
\limsup_{n\to\infty}\|f-\Pi_{n,m}\|^{1/n}_K\le\frac{\rho_\mu(K)}{R_{\mu,
m}}<1. \end{equation}
\end{itemize}
\end{thm}

 If, in the above theorem, we take $w_n(z)=z^n,
n\in\mathbb{N}$, it follows that $\mu\equiv\delta_0$,
$\Sigma=\{0\}$, $D_{\mu,m}=D_m$, and $R_{\mu,m}=R_m$; therefore
regaining Montessus de Ballore's theorem. If we take $\Sigma$ to be
a regular compact set and $\mu$ precisely its equilibrium measure we
obtain the same result as that of \cite{saff2}.

It is possible to prove an analog of \eqref{montessus} regardless
whether or not the function $f$ has exactly $m$ poles on
$D_{\mu,m}$. To this end, we need some additional definitions. Take
an arbitrary $\varepsilon > 0$ and define the open set
$J_\varepsilon$ as follows. For $n \geq m$, let $J_{n,\varepsilon}$
denote the $\varepsilon/6mn^2$-neighborhood of the set of zeros of
$Q_{n,m}$ and let $J_{m-1,\varepsilon}$ denote the
$\varepsilon/6m$-neighborhood of the set of poles of $f$ in
$D_{\mu,m}(f)$. Set $J_\varepsilon = \cup_{n \ge m -1}
J_{n,\varepsilon}$. We have $\sigma(J_{\varepsilon}) < \varepsilon$
and $J_{\varepsilon_1} \subset J_{\varepsilon_2}$ for $
\varepsilon_1 < \varepsilon_2$. For any set $B \subset {\mathbb{C}}$
we put ${B}(\varepsilon) = {B} \setminus J_{\varepsilon} $.

From these properties it readily follows that if
$\{\varphi_n\}_{n\in\mathbb{N}}$ converges uniformly to the function
$\varphi$ on ${K}(\varepsilon)$ for every compact ${K} \subset D$
and for each $\varepsilon >0$, then $\{\varphi_n\}_{n\in\mathbb{N}}$
converges in $\sigma$-content to $\varphi$ on compact subsets of
$D$.

Due to the normalization (\ref{normado}), for any compact set $K$ of
$\mathbb{C}$ and for every $\varepsilon >0$, there exist positive
constants $C_1, C_2$, independent of $n$, such that
\begin{equation} \label{desig} \|Q_{n,m}\|_{K} < C_1,
\qquad \min_{z \in {K}(\varepsilon)}|Q_{n,m}(z)| > C_2 n^{-2m},
\end{equation}
where the second inequality is meaningful when ${K}(\varepsilon)$ is
a non-empty set.

\begin{lem}\label{lema3}  Let the measure $\mu$ be
the asymptotic zero distribution of the sequence of interpolation
points given by $ \{w_n\}_{n\in\mathbb{N}}$. Then, for each compact
set $K\subset D_{\mu,m}$ and for each $\varepsilon>0$, it holds
\begin{equation}\label{contenido}
 \limsup_{n\to\infty} \|f-\Pi_{n,m}\|^{1/n}_{K(\varepsilon)} \le
\frac{\rho_\mu(K)}{R_{\mu,m}}<1.
\end{equation}
\end{lem}
\begin{proof} Fix a compact set $K\subset D_{\mu,m}$ and
$\varepsilon>0$. Let $Q_m(z)=\prod_{i=1}^\nu (z-z_i)$, where
$z_i,\,i=1,\dots,\nu\le m,$ are the poles of $f$ on $D_{\mu,m}$.
From the definition of multi-point Pad\'e approximants it follows
that the function
$$
\frac{Q_m\,Q_{n,m}\,f-Q_m\,P_{n,m}}{w_{n+1}}
$$
is holomorphic on $D_{\mu,m}\cup V$. Let $\eta>0$ arbitrarily small.
Let $\Gamma$ be a cycle contained in
$$
\left\{D_{\mu,m}\cup V \right\}\setminus
\left\{\overline{E_\mu(R_{\mu,m}-\eta)}\cup \Sigma\right\}
$$
homologous to $0$ in $D_{\mu,m}\cup V$, with winding number equal to
$1$ for all the poles $z_i,\,i=1,\dots,\nu,$ as well as the points
of $K\cup \Sigma$. This can be done because the compact set $
\overline{E_\mu(R_{\mu,m}-\eta)}\cup \Sigma $ is included in the
open set $D_{\mu,m}\cup V$. Then, for all $z\in\Gamma$, we have
\begin{equation}\label{acot}
\exp\{-P(\mu;z)\}>R_{\mu, m}-\eta.
\end{equation}
Applying the Cauchy integral formula, we obtain
\begin{multline*}
\dst\frac{[Q_m\,Q_{n,m}\,f-Q_m\,P_{n,m}](z)}{w_{n+1}(z)}=\dst\frac{1}{2\pi
i}\dst\int_{\Gamma}
\frac{[Q_m\,Q_{n,m}\,f](\zeta)}{w_{n+1}(\zeta)}\,
\frac{d\zeta}{\zeta-z}\\
\\ \dst
-\int_{\Gamma} \frac{[Q_m\,P_{n,m}](\zeta)}{w_{n+1}(\zeta)}\,
\frac{d\zeta}{\zeta-z}=\dst \frac{1}{2\pi i}\int_{\Gamma}
\frac{[Q_m\,Q_{n,m}\,f](\zeta)}{w_{n+1}(\zeta)}\,\frac{d\zeta}{\zeta-z},
\end{multline*}
for $z\in K$, where the second integral after the first equality is
zero due to the fact that the integrand is an analytic function
outside $\Gamma$ with a zero of multiplicity at least two at
infinity.

 Then,
for all $z\in K$
$$
[Q_m\,Q_{n,m}\,(f-\Pi_{n,m})](z)= \frac{w_{n+1}(z)}{2\pi
i}\int_{\Gamma} \frac{[Q_m\,Q_{n,m}\,f](\zeta)}{w_{n+1}(\zeta)}\,
\frac{d\zeta}{\zeta-z}
$$
and using \eqref{desig} we arrive at
\begin{equation}\label{desigualdad}
|f(z)-\Pi_{n,m}(z)|\le
C(K,\eta,\varepsilon)\,n^{2m}\frac{\|w_{n+1}\|_K}{\dst\min_{\zeta\in\Gamma}
|w_{n+1}(\zeta)|},
\end{equation}
for all $z\in K(\varepsilon)$, where the constant
$C(K,\eta,\varepsilon)$ is independent of $n$.

As the measure $\mu$ is the asymptotic zero distribution of the
sequence $\{w_n\}_{n\in\mathbb{N}}$, using \eqref{descent} and
\eqref{descent2}, we have
\begin{equation}\label{descenso1}
\limsup_{n\to\infty}|w_n(z_n)|^{1/n}\le e^{- P(\mu;z_0)},
\end{equation}
where the sequence $\{z_n\}_{n\in\mathbb{N}}$ tends to
$z_0\in\mathbb{C}$ as $n$ goes to infinity and
\begin{equation}\label{descenso2}
\lim_{n\to\infty}|w_n(z)|^{1/n}=e^{- P(\mu;z)},
\end{equation}
uniformly on compact subsets of $\mathbb{C}\setminus \Sigma$.

It follows from (\ref{descenso1}) that
\begin{equation}\label{descenso3}
\limsup_{n\to\infty}\|w_{n+1}\|_K^{1/n}\le\rho_{\mu}(K).
\end{equation}

Therefore, with the aid of (\ref{desigualdad}), (\ref{descenso2}),
(\ref{descenso3}), and (\ref{acot}), we obtain
$$
\limsup_{n\to\infty} \|f-\Pi_{n,m}\|_{K(\varepsilon)}^{1/n}\le
\frac{\rho_{\mu}(K)}{R_{\mu,m}-\eta}
$$
and formula \eqref{contenido} follows from the above expression
making $\eta$ go to $0$.
\end{proof}
Obviously, a proof of Theorem \ref{thwarner} may be obtained by
using Lemma \ref{lema3} and part ii) of Gonchar's lemma.

 Let us denote by $R^\prime_{\mu,m}$ the supremum of the numbers $r>r_0$ such
that the sequence $\{\Pi_{n,m}\}_{n\ge m}$ converges in
$\sigma$-content on compact subsets of $E_\mu(r)$. Due to
 Lemma \ref{lema3}, we have $0<R_{\mu,m}\le
R^\prime_{\mu,m}$ . On the other hand, from part i) of Gonchar's
lemma it follows that $f$ admits meromorhic continuation with at
most $m$ poles in the region where the sequence $\{\Pi_{n,m}\}_{n\ge
m}$ converges in $\sigma$-content. Therefore, it is clear that $
R_{\mu,m}=R^\prime_{\mu,m}$ and we could have used this property for
an alternative definition of $R_{\mu,m}$.

Finally, using \eqref{green2} for the compact set $K$, it is very
easy to see that the following corollary is essentially equivalent
to one of the implications in Theorem A.
\begin{cor}\label{teoremaa}
Let $K$ be a regular compact set and let $\{r_n\}_{n\ge m}$,
$r_n\in\mathcal{R}_{n-m,m}$, be a sequence of rational functions
satisfying
$$
\limsup_{n\to\infty} \sqrt[n]{\|f-r_n\|_K}\le\frac{ \capp
(K)}{\theta}<1.
$$
Then
$$
\sigma\mbox{-}\lim_{n\to\infty} r_n = f \quad \mbox{in}\quad
E_{\mu_K}(\theta),
$$
where $\mu_K$ is the equilibrium measure of $K$.
\end{cor}

%%%%%%%%%%%%%%%%%%%%%%%%%%%%%%%%%%%%%%%%%%%%%%%%%%%%%%%%%%%%%%%%%%%%%%%%%%%%%%%%%%%%%%%%%%%%%%%
%%%%%%%%%%%%%%%%%%%%%%%%%%%%%%%%%%%%%%%%%%%%%%%%%%%%%%%%%%%%%%%%%%%%%%%%%%%%%%%%%%%%%%%%%%%%%
%%%%%%%%%%%%%%%%%%%%%%%%%%%%   CONVERGENCIA  %%%%%%%%%%%%%%%%%%%%%%%%%%%%%%%%%%%%%%%%%%%%%%%%%%%%%%
%%%%%%%%%%%%%%%%%%%%%%%%%%%%%%%%%%%%%%%%%%%%%%%%%%%%%%%%%%%%%%%%%%%%%%%%%%%%%%%%%%%%%%%%%%%%%%%%%%%%%%%%%%%%%%%%%%%%%%%%%%%%%%%%%%%%%%%%%%%%%%%%%%%%%%%%%%%%%%%%%%%%%%%%%%%%%%%%%%%%%%%%%%%%%%%%

\section{Convergence in $\sigma$-content of row sequences}
\label{convergence}

Throughout the rest of the paper we maintain the notations
introduced above, i.e., $\Sigma$ is a compact set with connected
complement $\Omega$ and $f$ is an analytic function on a
neighborhood $V$ of $\Sigma$. For each nonnegative integers $n,m,\,
n\ge m$, $\Pi_{n,m}$ stands for the multipoint Pad\'e approximant of
type $(n,m)$ of $f$ with interpolation points given by $ w_{n+1}$.
Such points belong to $\Sigma$. Given a positive unit Borel measure
$\mu$ supported on $\Sigma$, recall that $\rho_\mu(K)=\|\exp\{
-P(\mu;\cdot)\}\|_K$ and $D_{\mu,m}$ is the set of $m$-meromorphy of
$f$ relative to $\mu $.

In this section we show that convergence in $\sigma$-content of the
approximants $\{\Pi_{n,m}\}$ is not possible outside $D_{\mu,m}$
provided we are not in the set of interpolation $\Sigma$. The
corresponding result for the classical case $\Sigma=\{0\}$ appears
in \cite{vavilov}.
\begin{thm}\label{divergencia}
Let the measure $\mu$ be the asymptotic zero distribution of the
sequence of interpolation points given by $
\{w_n\}_{n\in\mathbb{N}}$. Suppose that the sequence
$\{\Pi_{n,m}\}_{n\ge m}$ converges in $\sigma$-content on compact
subsets of a neighborhood of the point
$z_0\in\mathbb{C}\setminus\Sigma$. Then, $z_0\in D_{\mu,m}$.
\end{thm}

\begin{proof}
We argue by contradiction and divide the proof into three parts.
First, we reduce the result to finding a negative constant which
uniformly bounds from above a sequence of subharmonic functions
$\{h_n\}_{n\ge m}$. In the second part we study the boundary values
of $h_n, n\ge m,$ and obtain certain estimates for them which allow
us to define a harmonic majorant of all subharmonic functions
considered. This harmonic majorant has the desired properties and
with that we conclude the proof in part three.

1. Suppose that $z_0\not\in D_{\mu,m}$. As
$z_0\in\mathbb{C}\setminus\Sigma$ and $P(\mu;z)$ is a harmonic
function in that region, there exist
$z_1\in\mathbb{C}\setminus\Sigma$ and $\delta>0$ such that
$D(z_1,\delta)=\{z\in\mathbb{C}\,:\,|z-z_1|<\delta\}\subset
\mathbb{C}\setminus\Sigma$, $\exp\{-P(\mu;z)\}>R_{\mu,m}$ for all
$z\in \overline{D(z_1,\delta)}$, and the sequence
$\{\Pi_{n,m}\}_{n\ge m}$ converges in $\sigma$-content on compact
subsets of a neighborhood of $\overline{D(z_1,\delta)}$. Due to this
latter fact, for all $n\ge N_1$, it holds
$$
\sigma\left\{z\in\overline{D(z_1,\delta)}\,:\,\left|
\Pi_{n,m}(z)-\Pi_{n-1,m}(z)\right|\ge 1\right\}<\frac{\delta}{3}.
$$
Hence, there exists $\delta_1\in (\delta/3,\delta)$ such that
\begin{equation}\label{disco}
|\Pi_{n,m}(z)-\Pi_{n-1,m}(z)|<1,\quad |z-z_1|=\delta_1,\quad n\ge
N_1,
\end{equation}
where we have used the facts that $\sigma$-content is not increased
under circular projection and that the $\sigma$-content of a line
segment is half of its length.

As $z_0\not\in D_{\mu,m}$, we know that $R_{\mu, m}<+\infty$. Recall
that $V$ is the neighborhood of $\Sigma$ in which $f$ is analytic.
Set
$$
s_n\equiv
Q_{n,m}\,Q_{n-1,m}\,(\Pi_{n,m}-\Pi_{n-1,m})\in\mathcal{P}_{n},\quad
n\ge m.
$$
To reach a contradiction it is enough to prove that the functions
$$
h_n(z)=\frac{1}{n}\log |s_n(z)|+P(\mu;z)+\log R_{\mu,m},\quad n\ge
m,
$$
are uniformly bounded by a negative constant on a neighborhood $U$
of the set
$$
\{z\in\mathbb{C}\setminus V\,:\, \exp
\{-P(\mu;z)\}=R_{\mu,m}\}\not=\emptyset.
$$
In that case, there exist constants $q<1$  and $C_3>0$ verifying
\begin{equation}\label{esen}
 |s_n(z)|<C_3q^n,\quad n\ge N,\quad z\in U.
\end{equation}
Fix $\varepsilon>0$ and $K$ compact subset of $U$. Then, because of
\eqref{desig} and \eqref{esen}, we have
$$
\left|\Pi_{n,m}(z)-\Pi_{n-1,m}(z) \right|
=\left|\frac{s_n(z)}{Q_{n,m}(z)\,Q_{n-1,m}(z)}\right| \le
\frac{C_3}{C_2^2}\,q^n\,n^{4m},
$$
for $n\ge N$ and for all $z\in K(\varepsilon)$. As
$\sum_{n=1}^\infty q^n\,n^{4m}<+\infty$, the sequence
$\{\Pi_{n,m}\}_{n\ge m}$ converges uniformly on $K(\varepsilon)$.
Since $\varepsilon>0$ and $K$ are arbitrary, part i) of Gonchar's
lemma proves that the sequence $\{\Pi_{n,m}\}_{n\ge m}$ converges in
$\sigma$-content
 on compact subsets of $U$ to a meromorphic continuation of $f$ with at most $m$ poles, which
contradicts the definition of $R_{\mu,m}$.

2. As the functions $s_n$ are analytic, from the maximum principle
in $D(z_1,\delta_1)$, \eqref{desig},  and \eqref{disco}, it follows
that there exists a constant $C_4>0$ such that
$$
|s_n(z)|<C_4, \quad z\in \overline{D(z_1,\delta_1)}, \quad n\ge N_1.
$$
We also have
$$
P(\mu;z)+\log R_{\mu,m}<m<0,\quad z\in \overline{D(z_1,\delta_1)}.
$$
Then, there exists a negative constant $\alpha$ such that
\begin{equation}\label{negativo}
h_n(z)<\alpha<0,\quad z\in \overline{D(z_1,\delta_1)}, \quad n\ge
N_2\ge N_1.
\end{equation}

On the other hand, let $Q_m(z)=\prod_{i=1}^\nu (z-z_i)$, where
$z_i,\,i=1,\dots,\nu\le m,$ are the poles of $f$ in $D_{\mu,m}$.
 Let $W$ be a bounded open set with connected complement such that
$\overline{W}$ is regular,  $\Sigma\subset
W\subset\overline{W}\subset V,$ and  $\overline{W}\cap
\overline{D(z_1,\delta)}=\emptyset$. The existence of $W$ follows
from the Hilbert lemniscate theorem (see Theorem 5.5.8 in
\cite{ransford}). Consequently, $\partial W$ does not contain any
pole of the function $f$ in $D_{\mu,m}$ and the set
$D=\overline{\mathbb{C}}\setminus (\overline{W}\cup
\overline{D(z_1,\delta_1)})$ is connected.

Let $\eta>0$ be arbitrarily small. Let $\Gamma$ be a cycle contained
in
$$
\left\{D_{\mu,m}\cup V \right\}\setminus \left\{
\overline{E_\mu(R_{\mu,m}-\eta)}\cup \overline{W}\right\},
$$
homologous to $0$ in $D_{\mu,m}\cup V$ with winding number equal to
$1$ for all the poles $z_i,\,i=1,\dots,\nu,$ as well as the points
in $\overline{W}$. This can be done because the compact set $
\overline{E_\mu(R_{\mu,m}-\eta)}\cup \overline{W} $ is included in
the open set $D_{\mu,m}\cup V$. Now, we follow the arguments given
in the proof of Lemma \ref{lema3}. For all $z\in\Gamma$ it holds
\eqref{acot} and we can apply the Cauchy integral formula and the
principle of descent \eqref{descenso2} in an analogous way to obtain
$$
\limsup_{n\to\infty} |Q_{n,m}(z)\,(f(z)-\Pi_{n,m}(z))|^{1/n}\le
\frac{e^{- P(\mu;z)}}{R_{\mu,m}},
$$
uniformly on $\partial W$ since $\partial W\cap\Sigma=\emptyset$.
Then, using \eqref{desig}  we have
\begin{equation}\label{cota0}
\limsup_{n\to\infty} |s_n(z)|^{1/n}\le \frac{e^{-
P(\mu;z)}}{R_{\mu,m}},
\end{equation}
uniformly on $\partial W$. Therefore, for any given $\varepsilon>0$
there exists $N_{\varepsilon}\in\mathbb{N}, N_{\varepsilon}\ge N_2$,
such that
\begin{equation}\label{cota}
h_n(z)\le\varepsilon,\quad z\in\partial W,\quad n\ge
N_{\varepsilon}.
\end{equation}

3. Now, for each $\varepsilon>0$, a harmonic function
$\psi_\varepsilon$ is constructed as the solution on $
D=\overline{\mathbb{C}}\setminus
(\overline{W}\cup \overline{D(z_1,\delta_1)})
$ of the
Dirichlet problem given by the boundary values
$$
\psi_\varepsilon(z)=
  \begin{cases}
    \varepsilon, & \text{for}\,\, z\in \partial W, \\
    \alpha, & \text{for}\,\, z\in\partial D(z_1,\delta_1).
  \end{cases}
$$
Since the functions $h_n,\, n\in\mathbb{N},$ are subharmonic on $D$
and taking account of (\ref{negativo}) and (\ref{cota}), we see that
$\psi_\varepsilon$ is a harmonic majorant of $h_n$ for all $n\ge
N_\varepsilon$. Fix any compact set $K\subset D$. From the
two-constant theorem (see \cite{ransford}, Theorem $4.3.7$) it
follows that there exist positive constants $m_K$ and $M_K$
depending only on $K$ such that
$$
\psi_\varepsilon(z)\le\varepsilon\, m_K+\alpha\, M_K ,\quad z\in
K,\quad \varepsilon>0.
$$
Thus, for $\varepsilon>0$ sufficiently small, the function
$\psi_\varepsilon$ is bounded from above on $K$ by a negative
constant. The same property is then satisfied by the functions
$h_n,\, n\ge N_\varepsilon$. In particular, we can take $K$ in such
a way that it contains the open set $U$ of part 1, which finishes
the proof.
\end{proof}

%\begin{rem}\label{obs0}From Theorem \ref{divergencia} and Osgood's theorem
%(see Section 7.1 in \cite{remmert}) it follows that if the sequence
%$\{\Pi_{n,m}\}_{n\ge m}$ converges pointwise on a neighborhood of
%the point $z_0\in\mathbb{C}\setminus\Sigma$, then $z_0\in
%D_{\mu,m}$.
%\end{rem}

%\begin{rem}\label{obs1} Although in principle Theorem \ref{divergencia} does not give
%information about what happens on $\Sigma$ outside
%$\overline{D}_{\mu,m}$, notice that convergence in $\sigma$-content
%is not possible either in a neighborhood of
%$z_0\in\partial\Sigma\setminus \overline{D}_{\mu,m}$ since it would
%imply convergence in $\sigma$-content in a disk
%$D\subset\mathbb{C}\setminus\left(\overline{D}_{\mu,m}\cup\Sigma\right)$.
%\end{rem}

\begin{rem}\label{obs2}Let $\Lambda=\{n_k\}\subset\mathbb{N}$ be a
subsequence such that
\begin{equation}\label{sub}
\lim_{k\to\infty}\frac{n_{k+1}}{n_k}=1
\end{equation}
and suppose that $\{\Pi_{n,m}\}_{n\in\Lambda}$ converges in
$\sigma$-content on compact subsets of a neighborhood of the point
$z_0\in\mathbb{C}\setminus\Sigma$. Then, we can deduce that $z_0\in
D_{\mu,m}$ since all the arguments employed in the proof of Theorem
\ref{divergencia} are valid without any modification except to
obtain \eqref{cota0} where \eqref{sub} is strongly used.
\end{rem}

 We say that the interpolation points
$\alpha_{n,i},\, i=1,\dots,n,\,n\in\mathbb{N}$, (see
\eqref{interpolacion}) form a newtonian table if $w_n$ is divisor of
$w_{n+1}$ for all $n\in\mathbb{N}$, that is, if
$\alpha_{n,i}=\alpha_i,i=1,\dots,n$, $n\in\mathbb{N}$.

In the case that  $\Sigma=\{0\}$ pointwise divergence of the Pad\'e
approximants outside $\overline{D}_{m}$ has been proved  by Gonchar
in \cite{gonrows}. For multi-point Pad\'e approximants, this was
obtained by H. Wallin \cite{wallin} (excluding the set of
interpolation $\Sigma$) provided that the interpolation points form
a newtonian table and that the function $f$ has exactly $m$ poles in
$D_{\mu,m}$. This latter condition plays an essential role since it
assures the convergence of the denominators of the approximants
which is strongly used in \cite{wallin}.

We next prove that the sequence $\{\Pi_{n,m}\}_{n\ge m}$ diverges
pointwise outside the set $\overline{D}_{\mu,m}\cup \Sigma$ under the
only assumption that the interpolation points form a newtonian
table. To the best of our knowledge, proving such a result for
general tables of interpolation is an open problem, even for $m=0$.
In the particular case that the function $f$ is meromorphic with
simple poles, an easy argument based on the descent principle and
the residue theorem gives the desired divergence (cf.
\cite{wallin2}).

\begin{propo}\label{teorema2}
Let the measure $\mu$ be the asymptotic zero distribution of the
sequence of interpolation points given by $\{w_n\}_{n\in\mathbb{N}}$
which form a newtonian table. Then, the sequence
$\{\Pi_{n,m}\}_{n\ge m}$ diverges pointwise in $\mathbb{C}\setminus
\left(\overline{D}_{\mu,m}\cup \Sigma\right)$.
\end{propo}
\begin{proof} We may assume that  $R_{\mu, m}<+\infty$. Otherwise,
there is nothing left to prove. From the definition of multi-point Pad\'e
approximants and the fact that the table of point is newtonian it
follows that
$$
\frac{Q_{n,m}f-P_{n,m}}{w_{n+1}}\in \mathcal{H}(V),\quad\quad
\frac{Q_{n+1,m}f-P_{n+1,m}}{w_{n+1}}\in \mathcal{H}(V).
$$
Therefore
\begin{equation}\label{analitica}
\frac{Q_{n,m}P_{n+1,m}-Q_{n+1,m}P_{n,m}}{w_{n+1}}\in \mathcal{H}(V).
\end{equation}
The numerator of the fraction in \eqref{analitica} is a polynomial
of degree at most $n+1$ with zeros at the zeros of $w_{n+1}$.
Necessarily
$$
Q_{n,m}(z)P_{n+1,m}(z)-Q_{n+1,m}(z)P_{n,m}(z)=A_n w_{n+1}(z),
$$
or, equivalently
\begin{equation}\label{analitica2}
\Pi_{n+1,m}(z)-\Pi_{n,m}(z)=\frac{A_n
w_{n+1}(z)}{Q_{n,m}(z)Q_{n+1,m}(z)}.
\end{equation}
 Considering telescopic sums, it follows that the
sequence $\{\Pi_{n,m}\}_{n \geq m}$ converges or diverges with the
series
\[ \sum_{n \geq n_0} \frac{A_n
w_{n+1}(z)}{Q_{n,m}(z)Q_{n+1,m}(z)},
\]
where $n_0$ is chosen conveniently so that $Q_{n_0,m}(z) \neq 0$ at
the specific point under consideration. Set
$$
\frac{1}{R^*}=\limsup_{n\to\infty}|A_n|^{1/n}.
$$
Fix $\varepsilon>0$ and an arbitrary compact set $K\subset
E_{\mu}(R^*)$. By virtue of \eqref{desig} and \eqref{descenso1}, we
obtain
$$
\limsup_{n\to\infty}\left|\frac{A_n
w_{n+1}(z)}{Q_{n,m}(z)Q_{n+1,m}(z)}\right|^{1/n}\le\frac{\rho_\mu(K)}{R^*}<1,\quad
z\in K(\varepsilon),
$$
which implies that the sequence $\{\Pi_{n,m}\}_{n \geq m}$ converges
uniformly on $K(\varepsilon)$. Then  $\{\Pi_{n,m}\}_{n \geq m}$
converges in $\sigma$-content on compact subsets of $E_{\mu}(R^*)$
and, consequently, $R^*\le R_{\mu,m}$.

Let us consider $z_0\not\in \Sigma$ such that
$\rho_\mu(\{z_0\})>R^*$. Using now \eqref{desig} and
\eqref{descenso2} we have
\begin{equation}\label{analitica3}
\limsup_{n\to\infty}\left|\frac{A_n
w_{n+1}(z_0)}{Q_{n,m}(z_0)Q_{n+1,m}(z_0)}\right|^{1/n}\ge\frac{\rho_\mu(\{z_0\})}{R^*}>1.
\end{equation}
Hence, the sequence $\{\Pi_{n,m}(z)\}_{n \geq m}$ diverges at
$z=z_0$.

It remains to prove that $R^*\ge R_{\mu,m}$. Suppose that $R^*<
R_{\mu,m}$. Then, there exists $z_1\in D_{\mu,m}\setminus \Sigma$
with $\rho_\mu(\{z_1\})>R^*$. By the continuity of the potential
$P(\mu;\cdot)$ in a neighborhood of $z_1$, there exists a compact
disk $K_1=\overline{D(z_1,\delta)}\subset D_{\mu,m}\setminus \Sigma$
verifying $\rho_\mu(\{z\})>R^*$ for all $z\in K_1$.

Fix $\varepsilon>0$ with $\varepsilon<\delta$ so that
$K_1(\varepsilon)\not= \emptyset$. From \eqref{contenido} and
\eqref{analitica2} it follows
$$
\limsup_{n\to\infty}\left|\frac{A_n
w_{n+1}(z)}{Q_{n,m}(z)Q_{n+1,m}(z)}\right|^{1/n}\le\frac{\rho_\mu(K)}{R_{\mu,m}}<1,\quad
z\in K_1(\varepsilon),
$$
which is absurd in light of  \eqref{analitica3}.
\end{proof}

\begin{rem}
Notice that the proof of Proposition \ref{teorema2} does not make
use of the fact that the set $\Sigma$ is simply connected.
Therefore, that condition may be dropped in this case. Instead, we
must require the function $f$ to be analytic in $E_\mu(r)$ for some
$r>r_0$ since now we cannot apply Lemma \ref{lema}. The same can be
said about the other results concerning newtonian tables of
interpolation, that is, Proposition \ref{teorema1b} and Corollary
\ref{exactrate2} below.
\end{rem}
%%%%%%%%%%%%%%%%%%%%%%%%%%%%%%%%%%%%%%%%%%%%%%%%%%%%%%%%%%%%%%%%%%%%%%%%%%%%%%%%%%%%%%%%%%%%%%%
%%%%%%%%%%%%%%%%%%%%%%%%%%%%%%%%%%%%%%%%%%%%%%%%%%%%%%%%%%%%%%%%%%%%%%%%%%%%%%%%%%%%%%%%%%%%%
%%%%%%%%%%%%%%%%%%%%%%%%%%%%  RAZON DE CONVERGENCIA  %%%%%%%%%%%%%%%%%%%%%%%%%%%%%%%%%%%%%%%%%%%%%%%%%%%%%%
%%%%%%%%%%%%%%%%%%%%%%%%%%%%%%%%%%%%%%%%%%%%%%%%%%%%%%%%%%%%%%%%%%%%%%%%%%%%%%%%%%%%%%%%%%%%%%%%%%%%%%%%%%%%%%%%%%%%%%%%%%%%%%%%%%%%%%%%%%%%%%%%%%%%%%%%%%%%%%%%%%%%%%%%%%%%%%%%%%%%%%%%%%%%%%%%

\section{Meromorphic continuation}\label{exact}

We are ready for our main result.

\begin{thm}\label{teorema1}
Let the measure $\mu$ be the asymptotic zero distribution of the
sequence of interpolation points given by
$\{w_n\}_{n\in\mathbb{N}}$. Let $K$ be a regular compact set for
which the value $\rho_\mu(K)$ is attained at a point that does not
belong to the interior of $\Sigma$. Suppose that the function $f$ is
defined on $K$ and fulfills
\begin{equation}\label{hypo}
\limsup_{n\to\infty}\|f-\Pi_{n,m}\|^{1/n}_K\le\frac{\rho_\mu(K)}{R}<1.
\end{equation}
Then, $R_{\mu,m}\ge R$, that is, $f$ admits meromorphic continuation
with at most $m$ poles on the  set $E_\mu(R)$.
\end{thm}

\begin{proof}
Several arguments that will be used below are similar to those
employed in the proof of Theorem \ref{divergencia}. In particular,
we divide the proof into three parts along the same lines.

1. Let us suppose that $R_{\mu,m}<R$, hence we will come into
contradiction. It follows from (\ref{hypo}), using Corollary
\ref{teoremaa}, that the multi-point approximants $\Pi_{n,m}$
 converge in $\sigma$-content to $f$ on compact subsets of
 $E_{\mu_K}(\lambda_0)$, where
\begin{equation}\label{defi}
\lambda_0=\capp (K)\,\frac{R}{\rho_\mu(K)}>\capp (K).
\end{equation}
In case that $\rho_\mu(K)=0$ or $R=+\infty$ we have
$\lambda_0=+\infty$ and Theorem \ref{teorema1} trivially holds true.
Besides, from Theorem \ref{divergencia}, we have
$E_{\mu_K}(\lambda_0)\subset D_{\mu,m}\cup\Sigma$.

 Let $V$ be the neighborhood of $\Sigma$ in which $f$ is
analytic and set
$$
s_n\equiv
Q_{n,m}\,Q_{n-1,m}\,(\Pi_{n,m}-\Pi_{n-1,m})\in\mathcal{P}_n,\quad
n\ge m.
$$
Using the same argument that appears in the proof of Theorem
\ref{divergencia}, we see that in order to reach a contradiction it
is enough to prove that the functions
$$
h_n(z)=\frac{1}{n}\log |s_n(z)|+P(\mu;z)+\log R_{\mu,m}
$$
are uniformly bounded by a negative constant on a neighborhood $U$
of the set
$$
\{z\in\mathbb{C}\setminus V\,:\, \exp
\{-P(\mu;z)\}=R_{\mu,m}\}\not=\emptyset.
$$

2. For technical reasons that will
 become apparent below, we fix $\lambda>\capp(K)$, $\lambda<\lambda_0$, sufficiently close to $\capp(K)$
so that
\begin{equation}\label{constante}
2\,\log\left(\frac{\lambda}{\capp K}\right)<\,\frac{1}{4}\log\left(
\frac{R}{R_{\mu,m}}\right).
\end{equation}

From (\ref{hypo}) and \eqref{desig} it follows that
\begin{equation}\label{bernstein0}
\limsup_{n\to\infty} \|s_n\|^{1/n}_{K}\le\frac{\capp(K)}{\lambda_0}.
\end{equation}
Using the Bernstein-Walsh lemma (see Theorem $5.5.7$ in
\cite{ransford}), it holds
\begin{equation}\label{bernstein}
\limsup_{n\to\infty}
\|s_n\|^{1/n}_{F_\kappa}\le\frac{\kappa}{\lambda_0}<1,
\end{equation}
for all $\kappa$ with $\capp (K)<\kappa<\lambda_0$, where
$F_\kappa=\overline{E_{\mu_K}(\kappa)}$.

Fix $\delta>0$ such that
\begin{equation}\label{bernstein2}
\delta<\log\left(\frac{\lambda}{\capp(K)}\right)
\end{equation}
 and $\epsilon>0$ with $\epsilon<\lambda_0-\lambda$ and $\log (\lambda+\epsilon)<\log \lambda+\delta/2$.

Apply \eqref{bernstein} for $\kappa=\lambda+\epsilon$. Then, we have
$$
\frac{1}{n}\log |s_n(z)|\le
\log(\lambda+\epsilon)-\log\lambda_0+\delta/2,\quad n\ge
n_1(\delta),\quad z\in E_{\mu_K}(\lambda+\epsilon).
$$
Therefore,
$$
\frac{1}{n}\log |s_n(z)|\le \log\lambda-\log\lambda_0+\delta,\quad
n\ge n_1(\delta),\quad z\in E_{\mu_K}(\lambda+\epsilon).
$$
That is,
\begin{equation}\label{bernstein3}
h_n(z)\le  \log\lambda-\log\lambda_0+P(\mu;z)+\log R_{\mu,m}+\delta,
\end{equation}
for all $z\in E_{\mu_K}(\lambda+\epsilon)$ and $n\ge n_1(\delta)$.

Recall that $\Omega=\overline{\mathbb{C}}\setminus\Sigma$ is a
connected open set. Let $z_0\in K$ such that
$P(\mu,z_0)=-\log\rho_\mu(K)$. From hypotheses, we may assume that
$z_0\in\overline{\Omega}$ which implies that $z_0$ is a limit point
of $\Omega$ in the fine topology. As $K\subset
E_{\mu_K}(\lambda+\epsilon)$, there exist an open disk $B_0$
centered at $z=z_0$ and a point $z_1\in\Omega$ such that $z_1\in
B_0\subset\overline{B}_0\subset E_{\mu_K}(\lambda+\epsilon)$ and
$$
P(\mu;z_1)< -\log \rho_\mu(K)+\frac{1}{8}\log\left(
\frac{R}{R_{\mu,m}}\right).
$$
Then, there exists an open disk $B_1$ centered at $z=z_1$ with
$\overline{B}_1\subset B_0\cap \Omega$ and verifying
\begin{equation}\label{bernstein4}
P(\mu;z)< -\log \rho_\mu(K)+\frac{1}{4}\log\left(
\frac{R}{R_{\mu,m}}\right),\quad\mbox{for all}\quad z\in
\overline{B}_1,
\end{equation}
since $P(\mu;\cdot)$ is a continuous function on a neighborhood of
$z_1$.

Therefore, using (\ref{bernstein3}), (\ref{bernstein2}),
\eqref{defi}, (\ref{bernstein4}), and (\ref{constante}), for all
$z\in \overline{B}_1$ and $n\ge n_1(\delta)$, it holds
$$
\begin{array}{rcl}
h_n(z)&\le&\dst \log\lambda-\log\lambda_0+P(\mu;z)+\log
R_{\mu,m}+\delta\\ \\
& =&\dst\log\left(\frac{\lambda}{\capp (K)}\right)+P(\mu;z)+\log
R_{\mu,m}+\log\left(\frac{\capp(K)}{\lambda_0}\right)+\delta
\\ \\
&<&\dst 2\,\log\left(\frac{\lambda}{\capp(K)}\right)+P(\mu;z)+\log
R_{\mu,m}+\log\left(\frac{\capp(K)}{\lambda_0}\right)
\\ \\
&=&\dst
2\,\log\left(\frac{\lambda}{\capp(K)}\right)+P(\mu;z)+\log
\rho_\mu(K)-\log\left(\frac{R}{R_{\mu,m}}\right) \\ \\
&<&\dst 2\,\log\left(\frac{\lambda}{\capp(K)}\right)-\frac{3}{4}
\log\left(\frac{R}{R_{\mu,m}}\right)<
-\,\frac{1}{2}\,\log\left(\frac{R}{R_{\mu,m}}\right).
\end{array}
$$
Then, we obtain
\begin{equation}\label{bernstein5}
h_n(z)<\alpha<0,\quad\mbox{for all}\,\,\, z\in \overline{B}_1 \,\,\,
\mbox{and} \,\,\, n\ge n_1(\delta).
\end{equation}

On the other hand, let $Q_m(z)=\prod_{i=1}^\nu (z-z_i)$, where
$z_i,\,i=1,\dots,\nu\le m,$ are the poles of $f$ in $D_{\mu,m}$.
 Let $W$ be a bounded open set with connected complement such that
$\overline{W}$ is regular,  $\Sigma\subset
W\subset\overline{W}\subset V,$ and  $\overline{W}\cap
\overline{B}_1=\emptyset$. Consequently, $\partial W$ does not
contain any pole of the function $f$ in $D_{\mu,m}$ and the set
$D=\overline{\mathbb{C}}\setminus (\overline{W}\cup \overline{B}_1)$
is connected.

Reasoning as in the proof of Theorem \ref{divergencia}, we have
$$
\limsup_{n\to\infty}\frac{1}{n}\log|Q_{n,m}(z)\,(f-\Pi_{n,m})(z)|\le
-P(\mu;z)-\log R_{\mu,m},
$$
uniformly on $\partial W$. Hence, on account of \eqref{desig}, it
holds
\begin{equation}\label{epsilon0}
\limsup_{n\to\infty}\,h_n(z)\le 0,
\end{equation}
uniformly on $\partial W$. Then, fixed $\varepsilon>0$, we obtain
\begin{equation}\label{epsilon}
\,h_n(z)\le \varepsilon, \quad n\ge n_2(\varepsilon), \quad z\in
\partial W.
\end{equation}

3. For each $\varepsilon\ge 0$, a harmonic function
$\psi_\varepsilon$ is constructed as the solution on
$D=\overline{\mathbb{C}}\setminus (\overline{W}\cup \overline{B}_1)$
of the Dirichlet problem given by the boundary values
$$
\psi_\varepsilon(z)=
 \begin{cases}
    \varepsilon, & \text{for}\,\, z\in \partial W, \\
    \alpha, & \text{for}\,\, z\in\partial B_1.
  \end{cases}
$$
As the functions $h_n$ are subharmonic on $D\subset
\overline{\mathbb{C}}\setminus \Sigma$ and taking account of
(\ref{epsilon}) and (\ref{bernstein5}), we see that
$\psi_\varepsilon$ is a harmonic majorant of $h_n$ for all $n\ge
N=\max\{n_1(\delta),n_2(\varepsilon)\}$. Besides, the open set $U$
of part 1 verifies $U\subset D$ since
$E_{\mu_K}(\lambda_0)\setminus\Sigma\subset D_{\mu,m}$. The rest of
the proof is analogous to that of part $3$ of Theorem
\ref{divergencia}.
\end{proof}

%We wish to mention that
Theorem \ref{teorema1} was proved in \cite{grothmann} for the
particular case $m=0$ under the additional assumption that the sets
$K$ and $\Sigma$ coincide.

\begin{rem}\label{obs4} Due to the same reasons as in Remark \ref{obs2},
Theorem \ref{teorema1} may be proved under the slightly weaker
hypothesis that formula \eqref{hypo} holds true only for a
subsequence $\Lambda=\{n_k\}\subset\mathbb{N}$ verifying
\eqref{sub}. Condition \eqref{sub} is used here to obtain
\eqref{bernstein0} and \eqref{epsilon0}.
\end{rem}

%\begin{rem}\label{obs5}In some particular cases related to
%interpolation according to the equilibrium distribution of $\Sigma$
%it is known that geometric convergence in a compact subset contained
%in the interior of $\Sigma$ determines the region of meromorphic or
%analytic continuation of $f$ (cf. pp. 153-154 in \cite{walsh}).
%\end{rem}

In the case that the table of interpolation points is newtonian, we
can be more precise.

\begin{propo}\label{teorema1b}
Let the measure $\mu$ be the asymptotic zero distribution of the
sequence of interpolation points given by $\{w_n\}_{n\in\mathbb{N}}$
which form a newtonian table. Suppose that the function $f$ is
defined at a point $z\not\in\Sigma$ and fulfills
$$
\limsup_{n\to\infty}|f(z)-\Pi_{n,m}(z)|^{1/n}\le\frac{\rho_\mu(\{z\})}{R}<1.
$$
Then, $R_{\mu,m}\ge R$, that is, $f$ admits meromorphic continuation
with at most $m$ poles on the  set $E_\mu(R)$.
\end{propo}

\begin{proof}
From \eqref{analitica2} it follows that
\begin{equation}\label{An1}
|A_n|\le\left|\frac{Q_{n,m}(z)\,Q_{n+1,m}(z)
}{w_{n+1}(z)}\right|\,(|f(z)-\Pi_{n+1,m}(z)|+|f(z)-\Pi_{n,m}(z)|)
\end{equation}
 Hence, taking limits in
\eqref{An1} and using \eqref{desig} and \eqref{descenso2}, we obtain
\begin{equation}\label{An2}
\limsup_{n\to\infty}|A_n|^{1/n}\le\frac{1}{\rho_{\mu}(\{z\})}\,
\limsup_{n\to\infty}|f(z)-\Pi_{n,m}(z)|^{1/n}.
\end{equation}
Notice that $\rho_{\mu}(\{z\})>0$ since $z\not\in\Sigma$. From
\eqref{An2} and the proof of Proposition \ref{teorema2} it follows
that
$$
\frac{1}{R_{\mu,m}}=\limsup_{n\to\infty}|A_n|^{1/n}\le\frac{1}{R},
$$
which gives the result.
\end{proof}

The following corollary is a straightforward consequence of Theorem
\ref{teorema1} and the definition of $R_{\mu,m}$.
\begin{cor}\label{exactrate1}
Let the measure $\mu$ be the asymptotic zero distribution of the
sequence of interpolation points given by
$\{w_n\}_{n\in\mathbb{N}}$. Suppose that $f$ has exactly $m$ poles
on $D_{\mu,m}$. Let $K\subset D_{\mu,m}$ be a regular compact subset
not containing any pole of the function $f$ and for which the value
$\rho_\mu(K)$ is attained at a point that does not belong to the
interior of $\Sigma$. Then, it holds
$$
\limsup_{n\to\infty}\|f-\Pi_{n,m}\|^{1/n}_K=\frac{\rho_\mu(K)}{R_{\mu,
m}}<1.
$$
\end{cor}

If we use Proposition \ref{teorema1b} instead, we analogously obtain
the next result.

\begin{cor}\label{exactrate2}
Let the measure $\mu$ be the asymptotic zero distribution of the
sequence of interpolation points given by $\{w_n\}_{n\in\mathbb{N}}$
which form a newtonian table. Suppose that $f$ has exactly $m$ poles
$z_i,\,i=1,\dots,m,$ on $D_{\mu,m}$. Then,
$$
\limsup_{n\to\infty}|f(z)-\Pi_{n,m}(z)|^{1/n}=
\frac{\rho_\mu(\{z\})}{R_{\mu,m}}<1,
$$
for each $z\in
D_{\mu,m}\setminus\left(\Sigma\cup\{z_1,\dots,z_m\}\right)$.
\end{cor}
%%%%%%%%%%%%%%%%%%%%%%%%%%%%%%%%%%%%%%%%%%%%%%%%%%%%%%%%%%%%%%%%%%%%%%%%%%%%%%%%%%%%%%%%%%
%%%%%%%%%%%%%%%%%%%%%%%%%%%%%%%%%%%%%%%%%%%%%%%%%%%%%%%%%%%%%%%%%%%%%%%%%%%%%%%%%%%%%%%%%%%
%%%%%%%%%%%%%%%%%%%%%%%%%%%%%%%%%% BIBLIOGRAFIA %%%%%%%%%%%%%%%%%%%%%%%%%%%%%%%%%%%%%%%%%%%%%
%%%%%%%%%%%%%%%%%%%%%%%%%%%%%%%%%%%%%%%%%%%%%%%%%%%%%%%%%%%%%%%%%%%%%%%%%%%%%%%%%%%%%%%%%%%%%
%%%%%%%%%%%%%%%%%%%%%%%%%%%%%%%%%%%%%%%%%%%%%%%%%%%%%%%%%%%%%%%%%%%%%%%%%%%%%%%%%%%%%%%%%%%

\end{document}